\documentclass[12pt]{article}
\usepackage{amssymb}	% Allows use of AMS's mathematical symbols
\usepackage{amsmath}    % Allows use of AMS's mathematical commands
\usepackage{epsfig}
\usepackage{latexsym}

\newenvironment{proof}[1][Proof:]{\begin{trivlist} 
\item[\hskip \labelsep {\bfseries #1}]}{\end{trivlist}} 

\def\e{\bf e}
\def\b{\bf b}
\def\p{\bf p}
\def\q{\bf q}
\def\r{\bf r}
\def\v{\bf v}
\def\Q{\bf Q}
\def\R{\bf R}
\def\A{\bf A}
\def\RR{\mathbb{R}}
\def\c{\bf c}

\def\I{\bf I}
\def\y{\bf y}
\def\s{\bf s}
\def\x{\bf x}
\def\S{\bf S}
\def\X{\bf X}

\def\T{\rm T}

\newtheorem{algorithm}{Algorithm}[section]

\newtheorem{theorem}{Theorem}[section]
\newtheorem{lemma}{Lemma}[section]

\newtheorem{remark}{Remark}[section]
\begin{document}
\title{An Efficient Polynomial Interior-Point Algorithm for Linear Programming}
\author{Yaguang Yang\thanks{NRC, Office of Research, 21 Church Street, 
Rockville, 20850. yaguang.yang@verizon.net.}}
%\journalname{Journal of Optimization Theory and Applications}

\date{\today}

\maketitle    

\begin{abstract}
For interior-point algorithms in linear programming, it is well-known that the selection of the centering parameter is crucial for proving polynomility in theory and for efficiency in practice. However, the selection of the centering parameter is usually by heuristics and separate from the selection of the line-search step size. The heuristics are quite different while developing practically efficient algorithms, such as MPC, and theoretically efficient algorithms, such as short-step path-following algorithm. This introduces a dilemma that some algorithms with the best-known polynomial bound are least efficient in practice, and some most efficient algorithms may not be polynomial. In this paper, we propose a systematic way to optimally select the centering parameter and line-search step size at the same time, and we show that the algorithm based on this strategy has the best-known polynomial bound and may be very efficient in computation for real problems.
\end{abstract}


\newpage
\section{Introduction}

Interior-point method has been a matured discipline in mathematical programming. It has been the only topic in several research monographs published in 1990s \cite{wright97,ye97,rtv97}, and it is also included in some of the most cited books in mathematical programming \cite{LY08,NW99}. But there are still some fundamental problems that need to be answered \cite{todd02}. For example, the most successful interior-point algorithm in practice is MPC which has not been proved to be polynomial although a lot of effort has been made. In fact, MPC may not be polynomial \cite{cartis09}. Therefore, a concern for simplex method \cite{kleeMinty72} remains for the state-of-the-art interior-point algorithms, i.e., the state-of-the-art interior-point algorithms may not be polynomial \cite{yang14}. In a recent paper, Salahi, Peng, and Terlaky \cite{spt07} bridges the gap between theory and practical interior-point method. The paper proposes a variant of Mehrotra's algorithms. By introducing some safeguards, the authors show that their algorithm is polynomial. 

Another troublesome phenomenon in interior-point method is that some algorithms with best polynomial bound are least efficient in practice, and some most efficient algorithms may not show the existence of a polynomial bound \cite{wright97,cartis09}. The main reason leading to this dilemma is that the selection of the centering parameter is based on heuristics while developing interior-point algorithms. To develop algorithms with the best polynomial bound, some researchers use the heuristics with the sole purpose in mind to device algorithms easy to show the low polynomial bound without considering the efficiency in practice. To develop efficient algorithms in practice, other researchers focus on the heuristics which by intuition will generate good iterates but ignore the problem of proving a polynomial bound. 

A widely used shortcut in developing interior-point algorithms is to separate the selection of the centering parameter from the selection of the line-search step size \cite{spt07,kmm93,mar90,Mehrotra92,LMS91,zhang96,yang12}. This strategy makes the problem simple to deal with but has to use heuristics in the selection of the centering parameter. Therefore, this is not an optimal strategy.

In this paper, we propose a systematic way to optimally select the centering parameters and line-search step size at the same time, aiming at minimizing the duality gap in all iterations. We show that this algorithm will have the best-known polynomial bound even though the estimation is extremely conservative. We use some Netlib test problems to demonstrate that the proposed algorithm may be very efficient compared to some well-known implementation of the most efficient algorithm such as MPC.

The remaining of the paper is organized as follows. Section 2 describes the problem. Section 3 devises the algorithm that optimally selects the centering parameter and the line-search step size by minimizing the duality gap in all iterations. We also show in this section that the algorithm has the best-known polynomial bounds. Section 4 provides some numerical test result to show that the algorithm may be very efficient. The conclusion remarks are summarized in the last section.

\section{Problem Descriptions}

Consider the Linear Programming in the standard form:
\begin{eqnarray}
\min \hspace{0.05in} {\c}^{\T}{\x}, \hspace{0.15in} \mbox{\rm subject to} 
\hspace{0.1in}  {\A}{\x}={\b}, \hspace{0.1in} {\x} \ge 0,
\label{LP}
\end{eqnarray}
where ${\A} \in {\RR}^{m \times n}$, ${\b} \in {\RR}^m $, ${\c} \in {\RR}^n$ are given, and ${\x} \in {\RR}^n$  is the vector to be optimized. Associated with the linear programming is the dual programming that is also presented in the standard form:
\begin{eqnarray}
\max \hspace{0.05in} {\b}^{\T}{\y}, \hspace{0.15in} \mbox{\rm subject to} 
\hspace{0.1in}  {\A}^{\T}{\y}+{\s}={\c}, \hspace{0.1in} {\s} \ge 0,
\label{DP}
\end{eqnarray}
where dual variable vector ${\y} \in {\RR}^m$, and dual slack vector 
${\s} \in {\RR}^n$. 
Throughout the paper, for feasible solutions of (\ref{LP}) and (\ref{DP}), we will denote the duality gap by 
\begin{equation}
u=\frac{{\x}^{\T}\s}{n},
\label{duality}
\end{equation}
the $i$th component of ${\x}$ by ${x}_i$, the Euclidean norm of ${\x}$ by $\| {\x} \|$, the identity matrix of any dimension by $\I$, the vector of all ones with appropriate dimension by $\e$, the Hadamard (element-wise) product of two vectors ${\x}$ and $\s$ by ${\x} \circ \s$, the transpose of matrix $\A$ by $\A^{\T}$, a basis for the null space of $\A$ by $\hat{\A}$. To make the notation simple for block column vectors, we will denote, for example, a point in the primal-dual problem $[{\x}^{\T}, {\y}^{\T}, \s^{\T}]^{\T}$ by $({\x},{\y}, \s)$. We will denote the initial point of any algorithm by $({{\x}}^0,{{\y}}^0, {\s}^0)$, the corresponding duality gap by $\mu_0$, the point after the $k$th iteration by $({{\x}}^k,{{\y}}^k, {\s}^k)$, the corresponding duality gap by $\mu_k$, the optimizer by $({\x}^*, {\y}^*, \s^*)$, the corresponding duality gap by $\mu_*$. For ${\x} \in {\RR}^n$, we will denote a related diagonal matrix by ${\x} \in {\RR}^{n \times n}$ whose diagonal elements are components of the vector ${\x}$.

The central-path ${\cal C}$ of the primal-dual linear programming problem is parameterized by a scalar $\tau > 0$ as follows. For each interior point $({\x}, {\y}, \s) \in {\cal C}$ on the central path, there is a $\tau > 0$ such that 
\begin{subequations}
\begin{align}
\A{\x}=\b  \label{patha} \\
\A^{\T}{\y}+\s=\c \label{pathb} \\
({\x},{\s}) > 0 \label{pathd}  \\
{x}_i{s}_i = \tau, \hspace{0.1in} i=1,\ldots,n \label{pathc}.
\end{align}
\label{centralpath}
\end{subequations}
\noindent
As $\tau \rightarrow 0$, the central path $({\x}(\tau), {\y}(\tau), 
\s(\tau))$ represented by (\ref{centralpath}) approaches to a solution of (\ref{LP}) because (\ref{centralpath}) reduces to the KKT condition as $\tau \rightarrow 0$. 

To avoid the high cost in finding the central-path, all path-following algorithms search the optimizer along a central-path neighborhood. The central-path neighborhood considered in this paper is defined as a collection of points that satisfy the following conditions,
\begin{equation}
{\cal F}^o (\theta)=\lbrace({\x}, {\y}, {\s}): {\A}{\x}={\b}, \hspace{0.01in}
{\A}^{\T}{\y}+{\s}={\c}, \hspace{0.01in} ({\x},{\s})>0, \hspace{0.01in} \| {\x} \circ {\s} - \mu {\e} \| \le 
\theta \mu \rbrace,
\label{infeasible}
\end{equation}
where $\theta \in (0, 1)$ is a fixed constant. Throughout the paper, we make the following assumptions. 
\newline
{\bf Assumptions:}
\begin{itemize}
\item[1.] $\A$ is a full rank matrix.
\item[2.] ${\cal F}^o(\theta)$ is not empty.
\end{itemize}
Assumption 1 is trivial as $\A$ can always be reduced to meet this condition in polynomial operations. Assumption 2 implies the existence of a central path.

\section{Arc-Search Algorithm for Linear Programming}

Starting from any point $({\x}^0, {\y}^0, {\s}^0)$ in a central-path neighborhood that satisfies $({\x}^0,{\s}^0)>0$ and $\| {\x}^0 {\s}^0 - \mu_0 {\e} \| \le \theta \mu$, instead of searching along the central-path, which is difficult to find in practice, we consider searching along a line inside ${\cal F}^o(\theta)$ defined as follows:
\begin{equation}
({\x}(\alpha,\sigma), {\y}(\alpha,\sigma), {\s}(\alpha,\sigma) ):= ( {{\x}}^k- \alpha \dot{{\x}}(\sigma), 
{{\y}}^k - \alpha \dot{{\y}}(\sigma), {\s}^k- \alpha \dot{\s}(\sigma)),
\label{updateXLS}
\end{equation}
where $\alpha \in [0,1]$, $\sigma \in [0,1]$, and $(\dot{{\x}}(\sigma),\dot{{\y}}(\sigma),\dot{\s}(\sigma))$ is define by 
\begin{equation}
\left[
\begin{array}{ccc}
\A & 0 & 0\\
0 & \A^{\T} & \I \\
{\S}^k & 0 & {{\X}}^k
\end{array}
\right]
\left[
\begin{array}{c}
\dot{{\x}}(\sigma) \\ \dot{{\y}}(\sigma)  \\  \dot{\s}(\sigma)
\end{array}
\right]
=\left[
\begin{array}{c}
0 \\ 0 \\ {{{\x}}^k} \circ {{\s}^k} - \sigma \mu_k \e
\end{array}
\right].
\label{tmp}
\end{equation}
Since the search stays in ${\cal F}^o (\theta)$, as ${{\x}}^k \circ {\s}^k \rightarrow 0$, (\ref{duality}) implies that $\mu_k \rightarrow 0$; hence, the iterats will approach to an optimal solution of (\ref{LP}) because (\ref{centralpath}) reduces to KKT condition. 

We will use several results that can easily be derived from
(\ref{tmp}). To simplify the notations, we will drop the superscript and subscript $k$ unless a confusion may be introduced. The first two results are from \cite{wright97}.

\begin{lemma}
Let $(\dot{{\x}}(\sigma), \dot{{\y}}(\sigma), \dot{\s}(\sigma))$ be defined in (\ref{tmp}). Then, the following relations hold.
\begin{equation}
{\s}^{\T} \dot{{\x}}(\sigma) + {{\x}}^{\T} \dot{\s}(\sigma) = {{\x}}^{\T} {\s} -\sigma \mu n,
\label{simple}
\end{equation}
\begin{equation}
\mu(\alpha, \sigma)=\frac{{\x}(\alpha,\sigma)^{\T} \s(\alpha,\sigma) }{n} = \mu (1-\alpha(1 - \sigma)).
\label{obj}
\end{equation}
\end{lemma}

\begin{lemma}
Let $(\dot{{\x}}(\sigma), \dot{{\y}}(\sigma), \dot{s}(\sigma))$ be defined in (\ref{tmp}). Assume that 
$({\x}, {\y}, {\s}) \in {\cal F}^o(\theta)$. Then, the following relations hold.
\begin{equation}
{\A}{\x}(\alpha,\sigma) = {\b}, \hspace{0.1in} {A}^{\T} {\y}(\alpha,\sigma) +{\s}(\alpha,\sigma) = \c.
\end{equation}
\label{feasible}
\end{lemma}

Similar to the derivation of Lemma 3.5 in \cite{yang14}, we can establish the following lemma.
\begin{lemma}
Let $(\dot{{\x}}(\sigma), \dot{{\y}}(\sigma), \dot{\s}(\sigma))$ be defined in (\ref{tmp}). Then, the following relations hold.
\begin{subequations}
\begin{align}
\dot{{\x}}(\sigma) = \hat{\A}(\hat{\A}^{\T}{\S}{\X}^{-1}\hat{\A})^{-1}\hat{\A}^{\T}({\S}{\e}-\sigma {\X}^{-1} \mu \e):={\p}_x-\sigma {\q}_x, \\
\dot{\s}(\sigma) = {\A}^{\T}( {\A}{\X}{\S}^{-1}{\A}^{\T})^{-1} {\A} ({\X}\e-\sigma {\S}^{-1} \mu \e):={\p}_s-\sigma {\q}_s, \\
\dot{{\y}}(\sigma) = - ( {\A}{\X}{\S}^{-1}{\A}^{\T})^{-1} {\A} ({\X}{\e}-\sigma {\S}^{-1} \mu {\e}),
\end{align}
\label{dotXdotS}
\end{subequations}
where 
\begin{eqnarray} \nonumber
{\p}_x=\hat{\A}(\hat{\A}^{\T}{\S}{\X}^{-1}\hat{\A})^{-1}\hat{\A}^{\T}{\S}\e, \hspace{0.1in}
{\q}_x=\mu \hat{\A}(\hat{\A}^{\T}{\S}{\X}^{-1}\hat{\A})^{-1}\hat{\A}^{\T}{\X}^{-1} {\e}, \\ \nonumber
{\p}_s={\A}^{\T}( {\A}{\X}{\S}^{-1}{\A}^{\T})^{-1} {\A} {\X}{\e}, \hspace{0.1in}
{\q}_s=\mu {\A}^{\T}( {\A}{\X}{\S}^{-1}{\A}^{\T})^{-1} {\A} {\S}^{-1} {\e}.
\label{pxqxpsqs}
\end{eqnarray}
\end{lemma}
\begin{proof}
From the first two rows of (\ref{tmp}), we have, for some vector ${\v}$,
\begin{equation}
{\X}^{-1}\dot{{\x}}(\sigma) = {\X}^{-1}\hat{\A} {\v}, \hspace{0.1in} {\S}^{-1}\A^{\T} \dot{{\y}} (\sigma) +\S^{-1} \dot{\s}(\sigma) =0.
\label{tmp2}
\end{equation}
From the third row of (\ref{tmp}), we have,
\[
{\X}^{-1} \dot{{\x}}(\sigma) + \S^{-1} \dot{\s}(\sigma) = \e -\sigma \mu {\X}^{-1} \S^{-1} \e.
\]
Substituting the first two equations into the last equation and writing the result as a matrix form yield
\[
\left[ {\X}^{-1} \hat{\A}, -{\S}^{-1}\A^{\T} \right] \left[  \begin{array}{c} {\v} \\ \dot{{\y}}(\sigma) \end{array} \right]
={\e} -\sigma \mu {\X}^{-1} \S^{-1} {\e}.
\]
Since ${\A}$ is full rank, we have 
\[
\left[ \begin{array}{c}
(\hat{\A}^{\T}{\S}{\X}^{-1}\hat{\A})^{-1}\hat{\A}^{\T}{\S} \\ 
-(\A{\X}{\S}^{-1}\A^{\T})^{-1}{\A}{\X} \end{array} \right]
\left[ {\X}^{-1}\hat{\A}, -{\S}^{-1}{\A}^{\T} \right] = {\I}.
\]
This gives
\[
\left[  \begin{array}{c} {\v} \\ \dot{{\y}}(\sigma) \end{array} \right]
=\left[ \begin{array}{c}
(\hat{\A}^{\T}{\S}{\X}^{-1}\hat{\A})^{-1}\hat{\A}^{\T}{\S} \\ 
-(\A{\X}{\S}^{-1}\A^{\T})^{-1}{\A}{\X} \end{array} \right] ({\e} -\sigma \mu {\X}^{-1} {\S}^{-1} {\e}).
\]
Substituting this equation into (\ref{tmp2}) proves the result.
{\hfill \ensuremath{\Box}}
\end{proof}
Since
\[
\dot{{\x}}(\sigma) \circ \dot{\s}(\sigma)  = ({\p}_x-\sigma {\q}_x)  \circ ({\p}_s-\sigma {\q}_s) =
{\p}_x \circ {\p}_s -\sigma ({\q}_x\circ {\p}_s+{\p}_x\circ {\q}_s )+\sigma^2 {\q}_x\circ {\q}_s
:= {\p} -\sigma {\q} +\sigma^2 {\r},
\]
where 
\begin{equation}
{\p}={\p}_x \circ {\p}_s, \hspace{0.1in} {\q}={\q}_x\circ {\p}_s+{\p}_x \circ {\q}_s, \hspace{0.1in} {\r}={\q}_x \circ {\q}_s,
\label{pqr}
\end{equation}
to make sure that
$({\x}(\alpha,\sigma), {\y}(\alpha,\sigma), {\s}(\alpha,\sigma))$ stays in ${\cal F}^o (\theta)$, we need to find some $\bar{\alpha}$ such that for $\forall \alpha \in (0, \bar{\alpha}]$, the following inequality holds.
\begin{eqnarray} \nonumber
& & \| {\x}(\alpha,\sigma) \circ {\s}(\alpha,\sigma) - \mu(\alpha,\sigma) {\e} \|
= \| (1-\alpha) ({\x} \circ {\s} - \mu {\e})+ \alpha^2 ({\p} -\sigma {\q} +\sigma^2 {\r}) \|
\\ 
& \le & \theta \mu(\alpha, \sigma) = \theta \mu (1-\alpha(1 - \sigma)).
\label{proximity}
\end{eqnarray}
Assuming $\| {\x}  \circ {\s}  - \mu {\e} \| \le \theta \mu$, equation (\ref{proximity}) holds if
\begin{equation}
\| {\p} -\sigma {\q} +\sigma^2 {\r} \|^2  \le \frac{\theta^2 \sigma^2 \mu^2}{\alpha^2}.
\label{mainCondition}
\end{equation}
This is a quartic polynomial (in terms of $\sigma$) inequality constraint which can be written as
\begin{equation}
f(\sigma, \alpha):=a_4 \sigma^4 - a_3 \sigma^3 + 
\left(a_2 - \frac{\theta^2 \mu^2}{\alpha^2} \right)\sigma^2
- a_1 \sigma + a_0 \le 0, 
\label{mainCondition1}
\end{equation}
with 
\begin{equation}
a_0={\p}^{\T}{\p} \ge 0, \hspace{0.01in} a_1={\q}^{\T} {\p} +{\p}^{\T} {\q}, \hspace{0.01in} a_2={\p}^{\T}{\r}+{\r}^{\T}{\p}+{\q}^{\T}{\q}, \hspace{0.01in} a_3={\q}^{\T}{\r} + {\r}^{\T}{\q}, \hspace{0.01in} a_4={\r}^{\T}{\r} \ge 0.
\label{a04}
\end{equation} 
Here $a_i$, $i=0,1,2,3,4$, are all known constants since they are functions of ${\x}$ and ${\s}$ which are known at the beginning of every iteration. 

It is important to note that $f(\sigma, \alpha)$ is a monotonically increasing function of $\alpha$. Therefore, for any fixed $\sigma \in [0,1]$, if for some $\bar{\alpha}$, $ f(\sigma, \bar{\alpha}) \le 0$ holds, then $ f(\sigma, {\alpha}) \le 0$ holds for $\forall \alpha \in (0, \bar{\alpha}]$. Using the relation that $\| {\x}(\alpha,\sigma) \circ {\s}(\alpha,\sigma) - \mu(\alpha,\sigma) {\e} \| \le \theta \mu(\alpha, \sigma)$, we have $ {\x}_i(\alpha, \sigma){\s}_i(\alpha, \sigma) \ge (1-\theta) \mu(1-\alpha(1-\sigma)) >0$ for all $\forall \alpha \in (0, \bar{\alpha}]$. This means that $({\x}(\alpha,\sigma), {\s} (\alpha, \sigma)) > 0$ for all $\forall \alpha \in (0, \bar{\alpha}]$. Therefore, in the remaining discussions, we simply use $\alpha$ instead of $\bar{\alpha}$.

Assuming that the initial point $({\x}^0, {\y}^0, {\s}^0) \in {\cal F}^o(\theta)$, then in each iteration we want to minimize the duality gap $\mu(\alpha, \sigma)$ under the constraint that $({\x}(\alpha,\sigma), {\y}(\alpha, \sigma), {\s}(\alpha,\sigma)) \in {\cal F}^o(\theta)$.
Because of Lemma \ref{feasible}, the selection of $\alpha$ and $\sigma$ in each iteration is reduced to the following optimization problem.
\begin{eqnarray}\nonumber
\min_{\alpha, \sigma} & & \mu (1-\alpha(1-\sigma)) \\ 
s.t. & & 0 \le \alpha \le 1, \hspace{0.1in} 0 \le \sigma \le 1, \hspace{0.1in} f(\sigma, \alpha) \le 0.
\label{opt}
\end{eqnarray}
Since $0 \le \alpha \le 1$ and $0 \le \sigma \le 1$, we have $0 \le \alpha(1-\sigma) \le 1$. 
i.e., $0 \le (1-\alpha(1-\sigma)) \le 1$. This means that
$0 \le \mu(\alpha,\sigma) = \mu (1-\alpha(1-\sigma)) \le \mu$. 
Clearly, if $a_0=0$, then, the optimization problem has a solution of $\sigma=0$ and $\alpha=1$ with the 
objective funtion $\mu(\alpha, \sigma)=0$. One iteration will find the solution of (\ref{LP}).
Therefore, in the rest discussions, we do not consider this simple case. Instead, we assume that $a_0>0$ holds in all the iterations.
Let the Lagrange function be defined as follows.
\[
L = \mu(1-\alpha(1-\sigma))-\nu_1 \alpha -\nu_2 (1-\alpha) -\nu_3 \sigma -\nu_4 (1-\sigma)
+\nu_5 f(\sigma,\alpha),
\]
where $\nu_i$, $i=1,2,3,4,5$, are Lagrange multipliers. The KKT conditions for Problem (\ref{opt}) are as follows.
\begin{subequations}
\begin{align}\label{a}
\frac{\partial L}{\partial \alpha} = -(1-\sigma)\mu -\nu_1 +\nu_2 +2 \nu_5 
\frac{\sigma^2 \theta^2 \mu^2}{\alpha^3}=0, \\ \label{b}
\frac{\partial L}{\partial \sigma} = \alpha \mu -\nu_3+\nu_4 
+\nu_5\left( 4a_4\sigma^3-3a_3\sigma^2+2\left( a_2-\frac{\theta^2 \mu^2}{\alpha^2} \right)\sigma-a_1 \right) =0, \\ \label{c}
\nu_1 \ge 0, \hspace{0.01in} \nu_2 \ge 0, \hspace{0.01in}
\nu_3 \ge 0, \hspace{0.01in} \nu_4 \ge 0, \hspace{0.01in} \nu_5 \ge 0, \\ \label{d}
\nu_1 \alpha=0, \hspace{0.01in} \nu_2 (1-\alpha)=0, \hspace{0.01in} 
\nu_3 \sigma=0, \hspace{0.01in} \nu_4 (1-\sigma)=0, \hspace{0.01in} 
\nu_5 f(\sigma,\alpha) =0, \\ \label{e}
0 \le \alpha \le 1, \hspace{0.1in} 0 \le \sigma \le 1, \hspace{0.1in} f(\sigma, \alpha) \le 0.
\end{align}
\label{lagrange}
\end{subequations}
Relations in (\ref{lagrange}) can be simplified because of the following claims.
\begin{itemize}
\setlength{\itemindent}{2.8em}
\item[Claim 1]: $\alpha \neq 0$. Otherwise, $\mu(\alpha, \sigma)=\mu$ will be the maximum.
\item[Claim 2]: $\nu_1 = 0$ because of (\ref{d}).
\item[Claim 3]: $\sigma \neq 1$. Otherwise, $\mu(\alpha, \sigma)=\mu$ will be the maximum.
\item[Claim 4]: $\nu_4 = 0$ because of (\ref{d}).
\item[Claim 5]: $\sigma \neq 0$. Otherwise (\ref{mainCondition1}) does not hold since $a_0=p^{\T}p >0$ is assumed.
\item[Claim 6]: $\nu_3 = 0$ because of (\ref{d}).
\end{itemize}
Therefore, we can rewrite the KKT conditions as follows.
\begin{subequations}
\begin{align}\label{aa}
(\sigma -1)\mu +\nu_2 +2 \nu_5 
\frac{\sigma^2 \theta^2 \mu^2}{\alpha^3}=0, \\ \label{bb}
\alpha \mu +\nu_5\left( 4a_4\sigma^3-3a_3\sigma^2+2\left( a_2-\frac{\theta^2 \mu^2}{\alpha^2} \right)\sigma-a_1\right) =0, \\ \label{cc}
\nu_2 \ge 0, \hspace{0.01in} \nu_5 \ge 0, \\ \label{dd}
\nu_2 (1-\alpha)=0, \hspace{0.01in} 
\nu_5 f(\sigma,\alpha) =0, \\ \label{ee}
0 < \alpha \le 1, \hspace{0.1in} 0 < \sigma < 1, \hspace{0.1in} f(\sigma, \alpha) \le 0.
\end{align}
\label{KKT}
\end{subequations}
Notice that $f(\sigma, 1)<0$ cannot hold for all $\sigma \in (0,1)$, otherwise let $\sigma \rightarrow 0$, then $f(\sigma, 1) \rightarrow p^{\T}p >0$. Therefore, we divide our discussion into two cases.

Case 1: $f(\sigma, 1)=0$ has solution(s) in $\sigma \in (0,1)$. First, in view of the fact that $f(0,1)={\p}^{\T}{\p}>0$, it is straightforward to check that the smallest solution of $f(\sigma,1)=0$ in $\sigma \in (0,1)$ and $\alpha =1$ is a feasible solution and a candidate of the optimal solution that minimizes $\mu(\alpha, \sigma)=\mu(1-\alpha(1-\sigma))$ under all the constraints. Then, let us consider other feasible solutions which meet KKT condition but $\alpha<1$. Since $\alpha \neq 1$, we conclude that $\nu_2 =0$ from (\ref{dd}). From (\ref{aa}), we have
\[
\nu_5 =\frac{(1-\sigma)\alpha^3}{2 \mu \sigma^2 \theta^2} \neq 0.
\]
The last relation follows from the facts that $\alpha \neq 0$ and $\sigma \neq 1$. Substituting $\nu_5$ into (\ref{bb}) yields
\begin{eqnarray} 
\mu +\frac{(1-\sigma)\alpha^2}{2 \mu \sigma^2 \theta^2} \left( 4a_4\sigma^3-3a_3\sigma^2+2\left( a_2-\frac{\theta^2 \mu^2}{\alpha^2} \right)\sigma-a_1) \right) =0.
\label{tmp1}
\end{eqnarray}
Since $\nu_5 \neq 0$, from (\ref{dd}), we have
\[
f(\sigma,\alpha) = a_4 \sigma^4 - a_3 \sigma^3 + 
\left(a_2 - \frac{\theta^2 \mu^2}{\alpha^2} \right)\sigma^2
- a_1 \sigma + a_0=0,
\]
which gives, 
\begin{equation}
\frac{\theta^2 \mu^2 
}{\alpha^2 }\sigma^2 = a_4 \sigma^4 - a_3 \sigma^3 + 
a_2 \sigma^2- a_1 \sigma + a_0:=h(\sigma)> 0.
\label{hfunction}
\end{equation}
Substituting this relation into (\ref{tmp1}) and simplifying the result yield
\begin{equation}
g(\sigma):=(2a_4 -a_3)\sigma^4 +(2a_2-a_3)\sigma^3-3a_1\sigma^2+(4a_0+a_1)\sigma-2a_0=0. 
\label{gfunction}
\end{equation}
For all $\sigma \in (0,1)$ such that $g(\sigma)=0$, we can calculate $h(\sigma)=a_4 \sigma^4 - a_3 \sigma^3 + 
a_2 \sigma^2- a_1 \sigma + a_0$, and find
\begin{equation}
\alpha = \frac{\theta \mu \sigma}{\sqrt{h(\sigma)}}.
\label{alphaV}
\end{equation}
For all pairs $(\sigma, \alpha) \in (0,1) \times (0,1)$ obtained this way, they are candidates of the optimal solutions of (\ref{opt}).

Case 2: $f(\sigma, 1)>0$ for all $\sigma \in (0,1)$. For any fixed $\sigma$, since $f(\sigma, \alpha)$ is a monotonic increasing function of $\alpha$ and $f(\sigma, 0) = -\infty$, there exists an $\alpha \in (0,1)$ such that $f(\sigma, \alpha)=0$. It is easy to see that $\alpha \neq 1$ (otherwise the constraint $f(\sigma, \alpha) \le 0$  will not hold). Therefore, all arguments for $\alpha \neq 1$ in Case 1 apply here. Furthermore, in this case, we have a stronger condition than (\ref{hfunction}), i.e.,
\begin{equation}
\frac{\theta^2 \mu^2 
}{\alpha^2 }\sigma^2 = a_4 \sigma^4 - a_3 \sigma^3 + 
a_2 \sigma^2- a_1 \sigma + a_0:=h(\sigma)> f(\sigma,1)>0, \hspace{0.1in} \forall \sigma \in (0,1).
\label{case2}
\end{equation}
In view of the facts that $g(0)=-2a_0<0$ and $g(1)=2(a_4-a_3+a_2-a_1+a_0)=2h(1)>0$, 
$g(\sigma)=0$ has solution(s) in $\sigma \in (0,1)$.

For any candidate pair $(\sigma, \alpha)$ of the optimal solution obtained in Cases 1 and 2, we use (\ref{obj}) to calculate $\mu(\alpha, \sigma)$ for all candidate pairs. The smallest $\mu(\alpha, \sigma)$ among all candidate pairs $(\sigma, \alpha)$ is the solution of (\ref{opt}). Now we are ready to present the algorithm.

\begin{algorithm} {\ } \\ 
Data: $\A$, $\b$, $\c$, $\hat{\A}$.  Parameters: $\theta \in (0,1)$. 
Iinitial point: $({\x}^0, {\y}^0, {\s}^0) \in {\cal F}^0$,
and ${\mu}_{0}=\frac{{{\x}}^{{0}^{\T}}{\s}^{0}}{n}$.
\newline
{\bf for} iteration $k=0,1,2,\ldots$
\begin{itemize}
\setlength{\itemindent}{2.8em}
\item[Step 1:] Calculate ${\p}_x$, ${\q}_x$, ${\p}_s$, ${\q}_s$, $\dot{{\x}}(\sigma)$, $\dot{{\y}}(\sigma)$, and $\dot{\s}(\sigma)$ using (\ref{dotXdotS}); ${\p}$, ${\q}$, and ${\r}$ using (\ref{pqr}); $a_0$, $a_1$, $a_2$, $a_3$, and $a_4$ using (\ref{a04}).
\item[Step 2:] Select $\alpha$ and $\sigma$ as follows.
\begin{enumerate}
\item If $a_0 = 0$     \\
	set $\sigma=0$ and $\alpha=1$.
\item else $a_0 > 0$  
\begin{enumerate}
\item Solve $f(\sigma,1)=0$. If $f(\sigma,1)$ has solution(s) in $\sigma \in (0,1)$, the smallest solution $\sigma \in (0,1)$ and $\alpha=1$ is a candidate of optimal solution.
\item Solve $g(\sigma)=0$. If $g(\sigma)$ has solutions in $\sigma \in (0,1)$, calculate $h(\sigma)$ and $\alpha$ using (\ref{hfunction}) and (\ref{alphaV}); for each pair of $(\sigma, \alpha)$, if the pair meets $0< \sigma <1$ and $0< \alpha <1$, the pair is a candidate of solution.
\item Calculate $\mu(\alpha, \sigma)$ using (\ref{obj}) for all candidate pairs; select $\sigma$ and $\alpha$ that generate the smallest $\mu(\alpha, \sigma)$.
\end{enumerate}
\end{enumerate}
\item[Step 3:] Set $({\x}(k+1), {\y}(k+1), {\s}(k+1))=({\x}-\alpha \dot{{\x}}(\sigma), {\y} -\alpha \dot{{\y}}(\sigma), {\s}-\alpha \dot{\s}(\sigma))$.
\end{itemize}
{\bf end (for)} 
\label{mainAlgo}
\end{algorithm}

\begin{remark}
The most expensive computations are in Step 1, which involve matrix inverse and products of matrices and vectors. It is worthwhile to note that the update of ${\y}$ is not necessary but it is included. The computations in Step 2 involve the quartic polynomial solutions of $f(\sigma,1)$ and $g(\sigma)$ which are negligible \cite{Shmakov11}. The computational details for quartic solution are described in \cite{yang15}.
\end{remark}

\begin{remark}
In the proof of the polynomiality of the short-step path-following algorithm, the condition 
\begin{equation}
\| {\p} -\sigma {\q} +\sigma^2 {\r} \|  \le {\theta \sigma \mu}, 
\label{weakCondition}
\end{equation}
is proved to hold when $\sigma=1-0.4/\sqrt{n}$ and $\alpha=1$ are selected \cite[(5.14) in Lemma 5.5, (5.15) and (5.16) in Theorem 5.6]{wright97}. But this selection is obviously not as good as the selection in Step 2: 2. (a) of Algorithm \ref{mainAlgo}.
Clearly, if $\alpha=1$, (\ref{weakCondition}) is equivalent to (\ref{mainCondition}); if $\alpha<1$, the constraint (\ref{mainCondition}) is less restrict than (\ref{weakCondition}) (allows more choices of $\alpha$ and $\sigma$). Because of 
the additional choices in the relaxed constraint, and because of the choice of $(\sigma, \alpha)$ in Algorithm~\ref{mainAlgo} is optimal, the reduction of the duality gap in every iteration of Algorithm~\ref{mainAlgo} is more than the reduction in the short-step path-following algorithm (the latter is $\mu_{k+1} = (1-0.4/\sqrt{n}) \mu_k$). Notice that the polynomial bound for the short-step path-following algorithm is ${\cal O}(\sqrt{n} \log(\frac{1}{\epsilon}))$, the polynomial bound of Algorithm~\ref{mainAlgo} is at least the same as or better than ${\cal O}(\sqrt{n} \log(\frac{1}{\epsilon}))$.
\end{remark}

We summarize the discussion in this section into the following theorem.

\begin{theorem}
Algorithm ~\ref{mainAlgo} is convergent with the polynomial bound at least the same as or better than ${\cal O}(\sqrt{n} \log(\frac{1}{\epsilon}))$.
\label{mainTheorem}
\end{theorem}

\section{Implementation and numerical test}

Algorithm \ref{mainAlgo} is implemented in MATLAB and test is conducted for Netlib test problems. We provide the implementation details and discuss the test result in this section.

\subsection{implementation}

Algorithm \ref{mainAlgo} is presented in a simple form which is convenient for analysis. Some implementation details are provided here. 

First, to have a large step size, we need to have a large central-path neighborhood, therefore, parameter $\theta=0.99$ is used. Second, the program needs a stopping criterion to avoid an infinity loop, the code stops if 
\[
%\frac{\|r_B\|}{\max \lbrace 1, \| {\b} \| \rbrace }
%+\frac{\|r_C\|}{\max \lbrace 1, \| {\c} \|  \rbrace }+
\frac{ \mu }{\max \lbrace 1, \| {\c}^{\T}{\x} \|, \|{\b}^{\T}{\y}\|  \rbrace } 
< 10^{-8}
\] 
%or 
%\[
%\mu < 10^{-8}
%\]
holds, which is similar to the stopping criterion of {\tt linprog} \cite{zhang96}. 

Our experience shows when iterations approach an optimal point, some ${x}_i$ and/or ${s}_j$ approach to zero, which introduces large numerical error in the matrix inverses of (\ref{dotXdotS}). Therefore, the following alternative formulas are used to replace (\ref{dotXdotS}). Using the QR decomposition, we can write
\[
{\X}^{-0.5}{\S}^{0.5}\hat{\A} = {\Q}_1{\R}_1,
\]
where ${\Q}_1$ is an orthonormal matrix in ${\RR}^{n \times (n-m)}$, and ${\R}_1$ is an invertible triangle matrix in ${\RR}^{(n-m) \times (n-m)}$. Then, we have
\begin{eqnarray}
& & \hat{\A}\left( \hat{\A}^{\T}{\S}{\X}^{-1}\hat{\A} \right)^{-1} \hat{\A}^{\T}
\\ \nonumber
& = & {\X}^{0.5}{\S}^{-0.5} \left( {\X}^{-0.5}{\S}^{0.5}\hat{\A}\left( \hat{\A}^{\T}{\S}{\X}^{-1}\hat{\A} \right)^{-1} \hat{\A}^{\T} {\X}^{-0.5}{\S}^{0.5}\right) {\X}^{0.5}{\S}^{-0.5}
\\ \nonumber
& = & {\X}^{0.5}{\S}^{-0.5} {\Q}_1{\Q}_1^{\T} {\X}^{0.5}{\S}^{-0.5}.
\end{eqnarray}
Therefore,
\begin{equation}
{\p}_x = {\X}^{0.5}{\S}^{-0.5} {\Q}_1{\Q}_1^{\T} {\X}^{0.5}{\S}^{0.5}{\e},
\hspace{0.1in}
{\q}_x = {\mu} {\X}^{0.5}{\S}^{-0.5} {\Q}_1{\Q}_1^{\T} {\X}^{-0.5}{\S}^{-0.5}\e,
\label{pxqx}
\end{equation}
Similarly, we can write
\[
{\X}^{0.5}{\S}^{-0.5} {\A}^{\T} = {\Q}_2{\R}_2,
\]
where ${\Q}_2$ is an orthonormal matrix in ${\RR}^{n \times m}$, and ${\R}_2$ is an invertible triangle matrix in ${\RR}^{m \times m}$, 
\begin{eqnarray}
{\A}^{\T}\left({\A}{\S}^{-1}{\X} {\A}^{\T} \right)^{-1} {\A}
={\X}^{-0.5}{\S}^{0.5} {\Q}_2{\Q}_2^{\T} {\X}^{-0.5}{\S}^{0.5},
\end{eqnarray}
and
\begin{equation}
{\p}_s = {\X}^{-0.5}{\S}^{0.5} {\Q}_2{\Q}_2^{\T} {\X}^{0.5}{\S}^{0.5}{\e},
\hspace{0.1in}
{\q}_s = {\mu} {\X}^{-0.5}{\S}^{0.5} {\Q}_2{\Q}_2^{\T} {\X}^{-0.5}{\S}^{-0.5}\e,
\label{psqs}
\end{equation}

\begin{remark}
It is observed that formulas (\ref{pxqx}) and (\ref{psqs}) produce much more accurate result than (\ref{dotXdotS}) when iterations approach to the optimal solution. For sparse matrix $\A$, we can use sparse QR decomposition \cite{davis08} but we have not implemented yet.
\end{remark}

\subsection{Some Netlib test problems}

Numerical tests have been performed for linear programming problems
in Netlib library. For Netlib problems, \cite{cg06} has classified these problems into two categories: problems with strict interior-point and problems without strict interior-point. Though the newly developed
Matlab codes and other existing codes can solve problems without strict interior-point, we are most interested in the problems with strict interior-point that is assumed by all feasible interior-point methods. Among these problems, we only choose problems which are presented in standard form and their $\A$ matrices are full rank. The selected problems are solved using our Matlab function {\tt optimalAlphaSigma} and function {\tt linprog} in Matlab optimization toolBox. 

For several reasons, it is impossible to be completely fair in the comparison of the test results obtained by {\tt optimalAlphaSigma} and {\tt linprog}. First, there is no detail about the initial point selection in {\tt linprog}. Second, {\tt linprog} does not allow to start from user selected initial point other than the one provided by {\tt linprog}. Third, there is no information on what preprocessing is actually used before {\tt linprog} starts to run MPC, we only know from \cite{LMS91} that preprocessing ``generally increases computational efficiency, often substantial''. 

We compare the two codes simply by using the iteration numbers for the tested problem which are listed in table 1. Only two Netlib problems that are classified as problems with strict interior-point and are presented in standard form are not included in the table because our old PC computer used in the test does not have enough memory to handle problems of this size.

For all problems, {\tt optimalAlphaSigma} starts with ${\x}={\s}={\e}$. A preprocessing described in \cite{cg06} is used to find an initial point before Algorithm \ref{mainAlgo} runs. The initial point used in {\tt linprog} is said to be similar to the one used in \cite{LMS91} with some minor modifications (see \cite{zhang96}).

\begin{table}
\begin{center}
\caption{Iteration counts for test problems in Netlib and Matlab}
\begin{tabular}{|c|c|c|c|c|}
\hline 
        & \multicolumn{2}{l|} {iterations used by different algorithms}
        & 
        \\
 
Problem & optimalAlphaSigma & linprog &  source \\
\hline
AFIRO        & 4   & 7  &  netlib \\ \hline
blend        & 13  & 12  &  netlib \\ \hline
SCAGR25      & 5   & 16 &  netlib \\ \hline
SCAGR7       & 7   & 12 &  netlib \\ \hline
SCSD1        & 18  & 10 &  netlib \\ \hline
SCSD6        & 26  & 12 &  netlib \\ \hline
SCSD8        & 19  & 11 &  netlib \\ \hline
SCTAP1       & 17  & 17 &  netlib \\ \hline
SCTAP2       & 17  & 18 &  netlib \\ \hline
SCTAP3       & 18  & 18 &  netlib \\ \hline
SHARE1B      & 11  & 22 &  netlib \\ \hline
\end{tabular}
\end{center}
\end{table}
%\title{Iteration numbers of different Matlab codes}

This result is very impressive because {\tt optimalAlphaSigma} does not have a ``corrector step'' which is used by MPC and many other algorithms. Although corrector step is not as expensive as ``predictor step'', it still needs some substantial numerical operations.

\section{Conclusions}

In this paper, we have proposed a polynomial interior-point path-following algorithm that searches the optimizers in a neighborhood similar to the short-step algorithm. The algorithm is therefore polynomial with the best known complexity bound. But in every iteration, instead of small improvement in short-step algorithm, the algorithm minimizes the objective function (minimizes the duality gap), therefore, it achieves significantly better improvement in the neighborhood.
Preliminary numerical results on some Netlib problems show that the algorithm is very promising.

%\bibliographystyle{unsrt}
%\bibliography{Myrefs}

\end{document}